\documentclass[11]{amsart}
\usepackage{amsbsy,amssymb,pstricks,pst-node,mathrsfs,MnSymbol}
\usepackage{subfigure}
\usepackage[left]{lineno}
\usepackage{verbatim}
\usepackage{tikz}
\usetikzlibrary{positioning}
\theoremstyle{definition}

\newtheorem{theorem}{Theorem}[section]
\newtheorem{thm}[theorem]{Theorem}
\newtheorem{lem}[theorem]{Lemma}

\newtheorem{defn}{Definition}[section]

\begin{document}
	

	%
\title{Forbidden Minor Characterization of 3-fold-3-splitting of Binary Gammoid}\maketitle

\markboth{ Shital Dilip Solanki and S. B. Dhotre}{Excluded Minor Characterization of 3-fold-3-splitting of Binary Gammoid}\begin{center}\begin{large} Shital Dilip Solanki$^1$ and S. B. Dhotre$^2$\end{large}\\\begin{small}\vskip.1in\emph{
			1. Ajeenkya DY Patil University, Pune-411047, Maharashtra,
			India\\ 
			2. Department of Mathematics,
			Savitribai Phule Pune University,\\ Pune - 411007, Maharashtra,
			India}\\
		E-mail: \texttt{1. shital.solanki@adypu.edu.in, 2. dsantosh2@yahoo.co.in. }\end{small}\end{center}\vskip.2in
	\begin{abstract} 
	The {\it r-fold-n-point-splitting} operation is an important operation in Graph Theory defined by Slater \cite{slater r fold}. Later, Ghafari \cite{Ghafari_3fold} extended {\it 3-fold-n-point-splitting} operation in binary matroids and obtained the result for Eulerian matroids whose 3-fold is Eulerian. In this paper, we give another approach to extend {\it 3-fold-3-point-splitting} in binary matroids in terms of splitting and haracterize binary gammoid whose {\it 3-fold-3-point spitting} is binary gammoid.  	 
\end{abstract}\vskip.2in

\noindent\begin{Small}\textbf{Mathematics Subject Classification (2010)}:
	05C83,	05B35, 05C50      \end{Small}\vskip.2in
\vskip.25in

	\baselineskip 14truept 
\section{Introduction}\vskip0.5cm \noindent
Refer to Oxley \cite{ox}, for undefined concepts and terminologies. \\
A very well-known splitting operation in the graph is defined by Fleishner \cite{fl} as given below. \\
Let $H$ be a connected graph with a node $v$ with, a degree of $v$ at least $3$. Let $a=v_1v$ and $b=v_2v$ be two arcs incident at $v$. The graph $H_{a,b}$ is obtained by adding new node $v'$ adjacent to $v_1, v_2$ and removing arcs $a,b$ and adding arcs $a'=v_1v'$ and $b'=v_2v'$. The splitting operation on $H$ is the transformation from $H$ to $H_{x,y}$.

Later, this operation was then extended to binary matroid by Raghunathan et al. \cite{ttr} as a splitting using two elements. In \cite{mms1}, Shikare generalized splitting using $n$-elements.
\begin{defn}\cite{mms1} \label{def_splitting_in_3fold}
	Let $M$ be a binary matroid and $A$ be the matrix representation of $M$. Let $T$ be a subset of $E(M)$, with $|T|=n$ and $A_T$ be a matrix obtained by appending a row at the bottom of the matrix $A$ with entries one in the column corresponding to the elements in $T$ and zero everywhere else. Then $M(A_T)$ is called splitting matroid which is denoted by $M_T$ and the operation is the splitting operation using $n$-elements.
\end{defn}
\noindent The {\it n-point splitting} in graph is defined by Slater \cite{slater_point splitting in graph}. The definition of the {\it n-point splitting} is given below.\\
Let $u$ be a node in a graph $H$ with $deg(u)\geq (2n-2)$ and $K$ be a graph obtained by replacing $u$ with two adjacent nodes $u_1$ and $u_2$, then if a node $v$ is adjacent to $u$ in $H$ make $v$ adjacent to either $u_1$ or $u_2$ but not both in $K$ such that degree of $u_1$ and $u_2$ is at least $n$. The n-point splitting is the transformation from $H$ to $K$. \\
The extension of this operation is defined by Azadi \cite{azt} and is called an element-splitting operation.
\begin{defn}\label{def_element_sp_in_3fold}
	Let $M$ be a binary matroid with matrix representation $A$. Let $T$ be a subset of $E(M)$, with $|T|=n$. Obtain a matrix $A_T'$ by appending a row at the bottom with entries one in the column corresponding to the elements in $T$ and zero elsewhere and appending a column at the right corresponding to an element $a$ with entries zero everywhere except in the last row where it takes value one. The matroid $M(A_T')$ is called an element splitting matroid denoted by $M_T'$ and the element splitting operation is the transformation from $M$ to $M_T'$.
\end{defn}
\noindent The {\it r-fold-n-point-splitting} operation in Graph Theory is defined by Slater\cite{slater r fold}. The definition is given below. \\ 
Let $H$ be a graph, suppose $u$ be a node adjacent to nodes $v_1, v_2, \cdots, v_t$, where $t \geq n$. Obtain a graph $K$ from $H$ by replacing node $u$ by a complete graph $(K_r)$, on nodes $u_1, u_2, \cdots, u_r$, where $2 \leq r \leq n$. Make each node $v_i$ adjacent to exactly one $u_i$ such that $deg(u_i) \geq n$, for $i=1,2, \cdots r$. Then, the {\it r-fold-n-point-splitting} operation is the transformation from $H$ to $K$. \\
Concerning the above definition, Figure \ref{fig 4-fold} illustrates the {\it 4-fold-4-point-splitting} operation. 
\begin{figure}[h!]
	\centering
\unitlength 1mm 
\linethickness{0.4pt}
\ifx\plotpoint\undefined\newsavebox{\plotpoint}\fi 
\begin{picture}(128.448,58.127)(0,0)
	\put(8.409,10.091){\circle*{2.102}}
	\put(8.409,23.545){\circle*{2.102}}
	\put(8.409,34.582){\circle*{2.102}}
	\put(47.3,9.986){\circle*{1.734}}
	\put(47.3,23.44){\circle*{1.734}}
	\put(47.3,34.477){\circle*{1.734}}
	\put(25.253,54.895){\circle*{1.734}}
	\put(35.975,44.804){\circle*{1.734}}
	\put(27.881,25.463){\circle*{1.734}}
	\multiput(8.409,34.687)(.0336953957,.0409007004){496}{\line(0,1){.0409007004}}
	\multiput(25.122,54.974)(.0364782494,-.0337121262){608}{\line(1,0){.0364782494}}
	\put(47.301,34.477){\line(0,-1){24.491}}
	\put(47.301,9.986){\line(-1,0){38.156}}
	\multiput(9.145,9.986)(-.135145,.030032){7}{\line(-1,0){.135145}}
	\put(8.304,34.687){\line(0,-1){24.491}}
	\multiput(8.304,10.196)(.0433092337,.0337109711){449}{\line(1,0){.0433092337}}
	\multiput(27.75,25.332)(-.41597784,-.0335466){47}{\line(-1,0){.41597784}}
	\multiput(8.409,34.687)(.0699490861,-.033651182){278}{\line(1,0){.0699490861}}
	\multiput(24.701,54.974)(.033497448,-.321113464){91}{\line(0,-1){.321113464}}
	\multiput(46.985,34.477)(-.0733805507,-.0337153882){265}{\line(-1,0){.0733805507}}
	\multiput(27.645,25.647)(.31533804,-.03336911){63}{\line(1,0){.31533804}}
	\multiput(27.434,25.437)(.0438932074,-.0337284646){455}{\line(1,0){.0438932074}}
	\put(8.619,35.108){\line(1,0){38.997}}
	\multiput(35.949,44.673)(-.0336360579,-.0761015809){250}{\line(0,-1){.0761015809}}
	\put(6.937,7.358){\makebox(0,0)[cc]{$v_1$}}
	\put(4.835,23.44){\makebox(0,0)[cc]{$v_2$}}
	\put(4.52,34.372){\makebox(0,0)[cc]{$v_3$}}
	\put(23.65,58.127){\makebox(0,0)[cc]{$v_4$}}
	\put(37.63,47.196){\makebox(0,0)[cc]{$v_5$}}
	\put(50.454,35.633){\makebox(0,0)[cc]{$v_6$}}
	\put(50.98,23.861){\makebox(0,0)[cc]{$v_7$}}
	\put(49.613,9.46){\makebox(0,0)[cc]{$v_8$}}
	\put(27.54,23.02){\makebox(0,0)[cc]{$u$}}
	\put(85.877,10.091){\circle*{2.102}}
	\put(85.877,23.546){\circle*{2.102}}
	\put(85.877,34.582){\circle*{2.102}}
	\put(124.769,9.986){\circle*{1.734}}
	\put(124.769,23.44){\circle*{1.734}}
	\put(124.769,34.477){\circle*{1.734}}
	\put(102.721,54.895){\circle*{1.734}}
	\put(113.443,44.804){\circle*{1.734}}
	\multiput(85.877,34.687)(.0336955645,.0409012097){496}{\line(0,1){.0409012097}}
	\multiput(102.59,54.974)(.0364786184,-.0337121711){608}{\line(1,0){.0364786184}}
	\put(124.769,34.477){\line(0,-1){24.491}}
	\put(124.769,9.986){\line(-1,0){38.156}}
	\multiput(86.613,9.986)(-.135143,.03){7}{\line(-1,0){.135143}}
	\put(85.772,34.687){\line(0,-1){24.491}}
	\put(86.087,35.108){\line(1,0){38.997}}
	\put(84.406,7.358){\makebox(0,0)[cc]{$v_1$}}
	\put(82.303,23.44){\makebox(0,0)[cc]{$v_2$}}
	\put(81.988,34.372){\makebox(0,0)[cc]{$v_3$}}
	\put(101.119,58.127){\makebox(0,0)[cc]{$v_4$}}
	\put(115.098,47.196){\makebox(0,0)[cc]{$v_5$}}
	\put(127.922,35.633){\makebox(0,0)[cc]{$v_6$}}
	\put(128.448,23.861){\makebox(0,0)[cc]{$v_7$}}
	\put(127.081,9.46){\makebox(0,0)[cc]{$v_8$}}
	\put(105.349,31.139){\circle*{1.734}}
	\put(97.571,24.728){\circle*{1.734}}
	\put(113.338,25.148){\circle*{1.734}}
	\put(105.98,19.787){\circle*{1.734}}
	\put(99.962,26.173){\line(0,1){.105}}
	\multiput(97.439,24.701)(.041178164,.033592712){194}{\line(1,0){.041178164}}
	\multiput(105.428,31.218)(.044716444,-.033682516){181}{\line(1,0){.044716444}}
	\multiput(113.522,25.122)(-.047839746,-.033689962){156}{\line(-1,0){.047839746}}
	\multiput(106.059,19.866)(-.059855832,.033577662){144}{\line(-1,0){.059855832}}
	\put(97.439,24.701){\line(1,0){16.187}}
	\multiput(105.323,31.113)(.030032,-1.516626){7}{\line(0,-1){1.516626}}
	\multiput(102.59,55.079)(.03373987,-.29068198){81}{\line(0,-1){.29068198}}
	\put(105.323,31.534){\line(3,5){7.883}}
	\multiput(113.311,25.753)(.0441312782,.0337002488){262}{\line(1,0){.0441312782}}
	\multiput(113.522,25.227)(.20427559,-.03371539){53}{\line(1,0){.20427559}}
	\multiput(105.848,19.656)(.0676804671,-.033651182){278}{\line(1,0){.0676804671}}
	\multiput(85.351,10.511)(.0737301178,.033651182){278}{\line(1,0){.0737301178}}
	\multiput(85.877,23.545)(.32735092,.03303541){35}{\line(1,0){.32735092}}
	\multiput(85.877,34.477)(.0406533365,-.0336946573){287}{\line(1,0){.0406533365}}
	\put(105.533,17.028){\makebox(0,0)[cc]{$u_1$}}
	\put(95.968,22.915){\makebox(0,0)[cc]{$u_2$}}
	\put(102.17,31.954){\makebox(0,0)[cc]{$u_3$}}
	\put(114.363,23.23){\makebox(0,0)[cc]{$u_4$}}
	\put(26.068,5){\makebox(0,0)[cc]{$H$}}
	\put(105.218,5){\makebox(0,0)[cc]{$K$}}
\end{picture}
\caption{4-fold-4-point-splitting operation}
\label{fig 4-fold}
\end{figure}
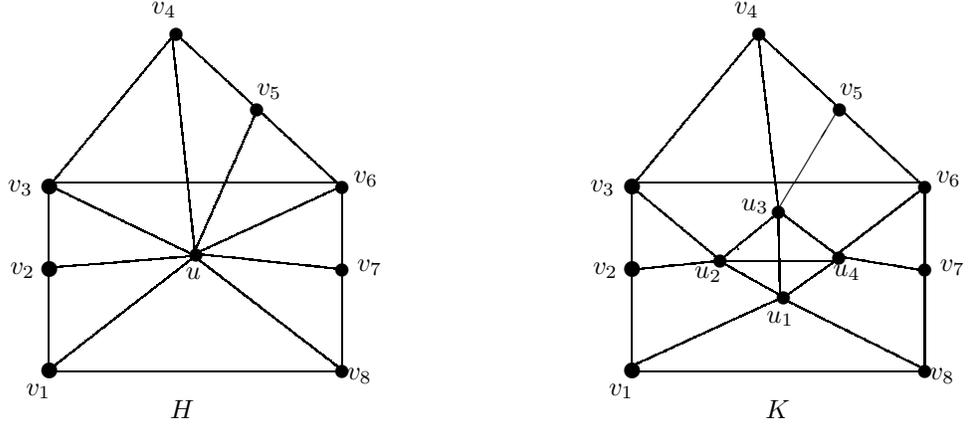

\noindent Figure \ref{graph 3fold} illustrates the construction of {\it 3-fold-3-point-splitting} graph from a given graph $H$. 
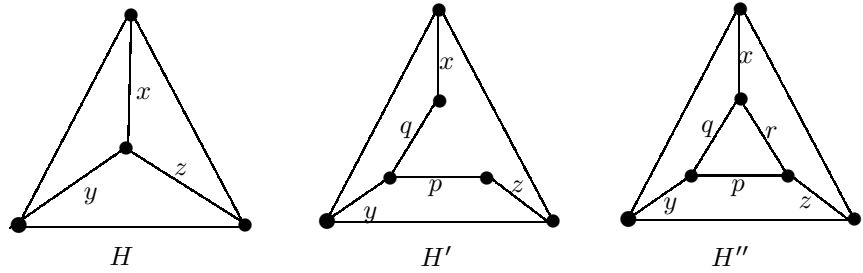
\begin{figure}[h!]
\centering
\unitlength 1mm 
\linethickness{0.4pt}
\ifx\plotpoint\undefined\newsavebox{\plotpoint}\fi 
\begin{picture}(120.381,39.832)(0,0)
	\put(8.409,10.091){\circle*{2.102}}
	\put(38.471,10.091){\circle*{1.734}}
	\put(23.335,38.051){\circle*{1.88}}
	\put(22.704,20.392){\circle*{1.784}}
	\multiput(22.704,37.63)(-.0337338371,-.064534297){430}{\line(0,-1){.064534297}}
	\put(8.199,9.881){\line(1,0){30.272}}
	\multiput(38.471,9.881)(-.0337109711,.0636762788){449}{\line(0,1){.0636762788}}
	\multiput(23.335,38.471)(-.0323424,-1.3745504){13}{\line(0,-1){1.3745504}}
	\multiput(22.915,20.602)(-.0486632781,-.0337398728){324}{\line(-1,0){.0486632781}}
	\multiput(22.494,20.392)(.0532301396,-.0336899618){312}{\line(1,0){.0532301396}}
	\put(49.403,10.721){\circle*{2.102}}
	\put(79.466,10.721){\circle*{1.734}}
	\put(64.329,38.681){\circle*{1.88}}
	\put(70.636,16.397){\circle*{1.784}}
	\put(57.812,16.397){\circle*{1.784}}
	\multiput(63.699,38.261)(-.0337348837,-.0645348837){430}{\line(0,-1){.0645348837}}
	\put(49.193,10.511){\line(1,0){30.272}}
	\multiput(79.466,10.511)(-.0337126949,.0636770601){449}{\line(0,1){.0636770601}}
	\multiput(63.698,26.278)(-.033725994,-.05508579){187}{\line(0,-1){.05508579}}
	\multiput(71.056,16.398)(.044135711,-.033682516){181}{\line(1,0){.044135711}}
	\multiput(57.602,16.398)(-.048781576,-.033682516){181}{\line(-1,0){.048781576}}
	\put(64.119,39.102){\line(0,-1){12.614}}
	\put(64.38,26.645){\circle*{1.784}}
	\put(57.812,16.503){\line(1,0){13.139}}
	\put(89.451,10.932){\circle*{2.102}}
	\put(119.514,10.932){\circle*{1.734}}
	\put(104.377,38.892){\circle*{1.88}}
	\put(110.684,16.608){\circle*{1.784}}
	\put(97.861,16.608){\circle*{1.784}}
	\multiput(103.747,38.472)(-.0337348837,-.0645348837){430}{\line(0,-1){.0645348837}}
	\put(89.241,10.722){\line(1,0){30.272}}
	\multiput(119.514,10.722)(-.0337126949,.0636770601){449}{\line(0,1){.0636770601}}
	\multiput(103.747,26.489)(-.033727273,-.055085561){187}{\line(0,-1){.055085561}}
	\multiput(111.105,16.608)(.044132597,-.033679558){181}{\line(1,0){.044132597}}
	\multiput(97.65,16.608)(-.048779006,-.033679558){181}{\line(-1,0){.048779006}}
	\put(104.167,39.313){\line(0,-1){12.614}}
	\put(104.428,26.855){\circle*{1.784}}
	\put(97.86,16.713){\line(1,0){13.139}}
	\multiput(104.167,27.119)(.033570018,-.053932161){191}{\line(0,-1){.053932161}}
	\put(25,27.645){\makebox(0,0)[cc]{$x$}}
	\put(17.974,14){\makebox(0,0)[cc]{$y$}}
	\put(29.957,17.5){\makebox(0,0)[cc]{$z$}}
	\put(65.275,31.639){\makebox(0,0)[cc]{$x$}}
	\put(55.184,11.5){\makebox(0,0)[cc]{$y$}}
	\put(74.735,15.241){\makebox(0,0)[cc]{$z$}}
	\put(63.803,15){\makebox(0,0)[cc]{$p$}}
	\put(59.914,22.704){\makebox(0,0)[cc]{$q$}}
	\put(108.476,22.494){\makebox(0,0)[cc]{$r$}}
	\put(104.062,15){\makebox(0,0)[cc]{$p$}}
	\put(99.962,22.915){\makebox(0,0)[cc]{$q$}}
	\put(105.113,32.48){\makebox(0,0)[cc]{$x$}}
	\put(95.022,12.6){\makebox(0,0)[cc]{$y$}}
	\put(112.996,13.139){\makebox(0,0)[cc]{$z$}}
	\put(22.074,6){\makebox(0,0)[cc]{$H$}}
	\put(63.909,6){\makebox(0,0)[cc]{$H'$}}
	\put(103.01,6){\makebox(0,0)[cc]{$H''$}}
\end{picture}
\caption{3-fold-3-point splitting.}
\label{graph 3fold}
\end{figure}
Let the graphs $H$, $H'$ and $H''$ be as shown in the Figure \ref{graph 3fold}. Then, the graph $H''$ is {\it 3-fold-3-point splitting} of $H$. By applying element splitting with respect to $\{x,y\}$ and then with respect to $x$ on $M(H)$ we get $M(H')$. And, by adding an element $r$ to $M(H')$ such that it forms a circuit with $\{p,q\}$, we get $M(H'')$, where the elements $p,q$ are added during two element splitting operations. The {\it 3-fold-n-point-splitting} operation is extended by Ghafari \cite{Ghafari_3fold} to binary matroids. He defined {\it 3-fold-n-point-splitting} operation in terms of an element splitting operation. The definition is given below.
\begin{defn}
	Let $M$ be a binary matroid with matrix representation $A$ and $T$ be a proper subset of the cocircuit of $M$. Suppose $\phi \neq T' \subset T$. Let $B=(A_T')_{T'}$, then by adding an element $r$ to $M(B)$, we get {\it 3-fold-n-point-splitting} of matroid $M$ and it is denoted by $M''$. An element $r$ is added such that it forms a circuit with $\{p,q\}$, where the elements $p,q$ are added during two element splitting operations. 
\end{defn}
\noindent Thus $E(M'')=E(M) \cup \{p,q,r\}$. \\
In this paper, we define {\it 3-fold-3-point-splitting} operation in terms of splitting. The splitting operation does not preserve the properties of matroid like graphicness, connectedness, cographicness, etc. In \cite{ymb_Gammoid, ymb1, gm, mms}, various researchers have studied splitting operations for these properties. Solanki et al. \cite{sds_Gammoid} characterized binary gammoid whose splitting is binary gammoid as given in the following theorem.
\begin{thm} \label{thm gam to gam wrt 3 elt in3fold}
	Let $M$ be a binary gammoid then for $T \subseteq E(M)$ with $|T|=3$, $M_T$ is a binary gammoid if and only if $M$ do not contain $M(G_i)$ as a minor, where the graph $G_i$ is shown in Figure \ref{fig gam to gam wrt 3 element in3fold}, for $i=1,2,3$. 
\end{thm}
\begin{figure}[h!]
	\centering
	\unitlength 1mm 
	\linethickness{0.4pt}
	\ifx\plotpoint\undefined\newsavebox{\plotpoint}\fi 
	\begin{picture}(82.015,31.5)(0,0)
		\put(25.56,8.438){\circle*{1.5}}
		\put(15.56,23.438){\circle*{1.5}}
		\put(6.06,8.469){\circle*{1.5}}
		\put(5.908,8.487){\line(1,0){19.534}}
		\multiput(15.345,23.588)(.0336731392,-.0485889968){309}{\line(0,-1){.0485889968}}
		\multiput(15.493,23.291)(-.0336655052,-.0512787456){287}{\line(0,-1){.0512787456}}
		\qbezier(15.642,23.44)(27.014,21.135)(25.602,8.426)
		\qbezier(5.98,8.574)(3.75,20.764)(15.196,23.44)
		\qbezier(15.345,23.737)(9.547,28.048)(15.642,29.386)
		\qbezier(15.642,29.386)(21.662,28.419)(15.493,23.588)
		\put(52.912,8.73){\circle*{1.5}}
		\put(42.912,23.73){\circle*{1.5}}
		\put(33.412,8.76){\circle*{1.5}}
		\put(33.26,8.779){\line(1,0){19.534}}
		\multiput(42.697,23.88)(.0336731392,-.0485889968){309}{\line(0,-1){.0485889968}}
		\multiput(42.845,23.583)(-.0336655052,-.0512787456){287}{\line(0,-1){.0512787456}}
		\qbezier(42.994,23.731)(54.366,21.427)(52.954,8.717)
		\qbezier(33.332,8.866)(31.102,21.055)(42.548,23.731)
		\qbezier(52.953,8.723)(59.94,9.541)(56.224,4.71)
		\qbezier(56.224,4.71)(51.021,1.811)(52.953,8.723)
		\put(80.561,8.73){\circle*{1.5}}
		\put(70.561,23.73){\circle*{1.5}}
		\put(61.061,8.76){\circle*{1.5}}
		\put(60.909,8.779){\line(1,0){19.534}}
		\multiput(70.346,23.88)(.0336731392,-.0485889968){309}{\line(0,-1){.0485889968}}
		\multiput(70.494,23.583)(-.0336655052,-.0512787456){287}{\line(0,-1){.0512787456}}
		\qbezier(70.643,23.731)(82.015,21.427)(80.603,8.717)
		\qbezier(60.981,8.866)(58.751,21.055)(70.197,23.731)
		\qbezier(60.981,8.723)(72.575,.25)(80.603,9.02)
		\put(15,1){\makebox(0,0)[cc]{$G_1$}}
		\put(42.5,1){\makebox(0,0)[cc]{$G_2$}}
		\put(70.75,1){\makebox(0,0)[cc]{$G_3$}}
		\put(4.5,19.5){\makebox(0,0)[cc]{$x$}}
		\put(15.25,31.5){\makebox(0,0)[cc]{$b$}}
		\put(26.25,19.25){\makebox(0,0)[cc]{$c$}}
		\put(34,21.75){\makebox(0,0)[cc]{$x$}}
		\put(53.25,20){\makebox(0,0)[cc]{$c$}}
		\put(56.5,2.25){\makebox(0,0)[cc]{$b$}}
		\put(61.25,20.75){\makebox(0,0)[cc]{$x$}}
		\put(81.5,19.75){\makebox(0,0)[cc]{$b$}}
		\put(71,7){\makebox(0,0)[cc]{$c$}}
	\end{picture}

	\caption{Forbidden minors for binary gammoid whose splitting using three elements is binary gammoid.}
	\label{fig gam to gam wrt 3 element in3fold}
\end{figure}
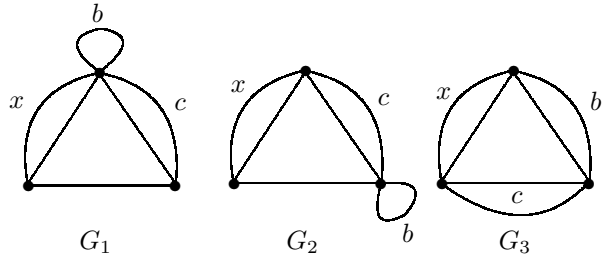
{\bf Abbreviation : } We use abbreviation {\it 3-fold} instead of {\it 3-fold-3-point-splitting}. 
\noindent In this paper, we characterize binary gammoid whose {\it 3-fold} is binary gammoid. We prove the following theorem.
\begin{thm}
Let $M$ be a binary gammoid then {\it 3-fold} of $M$ is binary gammoid if and only if $M$ does not contain $M(G_4)$ as a minor, where $G_4$ is shown in the Figure \ref{graph 3fold gammoid minor}.
\begin{figure}[h!]
\centering
\unitlength 1mm 
\linethickness{0.4pt}
\ifx\plotpoint\undefined\newsavebox{\plotpoint}\fi 
\begin{picture}(56.75,41.625)(0,0)
	\put(18,23.75){\circle*{2.5}}
	\put(55.5,23.75){\circle*{2.5}}
	\put(17.75,23.75){\line(1,0){38.5}}
	\qbezier(17.75,24)(37.625,41.625)(55,23.75)
	\qbezier(55,23.75)(36.75,6.75)(18.5,23.75)
	\put(36,36.75){\makebox(0,0)[cc]{$x$}}
	\put(37,26.25){\makebox(0,0)[cc]{$y$}}
	\put(36.5,4.5){\makebox(0,0)[cc]{$G_4$}}
\end{picture}

\caption{Excluded minor for a binary gammoid whose {\it 3-fold} is gammoid.}
\label{graph 3fold gammoid minor}
\end{figure}
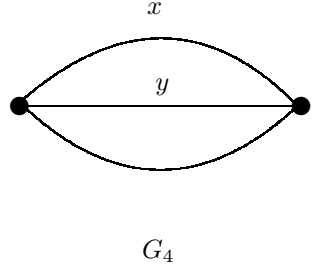
\end{thm}

\section{{Another approach to 3-fold}}  
In this section, we define {\it 3-fold} operation using the concept of splitting operation. Figure \ref{graph 3fold splitting} illustrate {\it 3-fold} operation using splitting. 
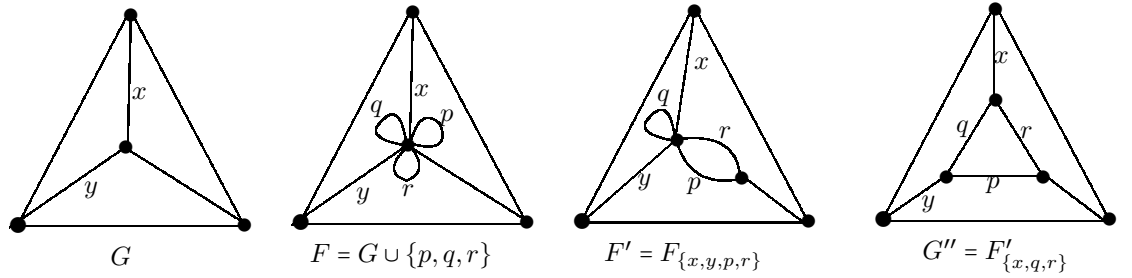
\begin{figure}[h!]
	\centering
\unitlength 1mm 
\linethickness{0.4pt}
\ifx\plotpoint\undefined\newsavebox{\plotpoint}\fi 
\begin{picture}(154.332,39.727)(0,0)
	\put(8.409,10.091){\circle*{2.102}}
	\put(38.471,10.091){\circle*{1.734}}
	\put(23.335,38.051){\circle*{1.88}}
	\put(22.704,20.392){\circle*{1.784}}
	\multiput(22.704,37.63)(-.0337325581,-.0645325581){430}{\line(0,-1){.0645325581}}
	\put(8.199,9.881){\line(1,0){30.272}}
	\multiput(38.471,9.881)(-.0337104677,.063674833){449}{\line(0,1){.063674833}}
	\multiput(23.335,38.471)(-.0323077,-1.3745385){13}{\line(0,-1){1.3745385}}
	\multiput(22.915,20.602)(-.0485138462,-.0336369231){325}{\line(-1,0){.0485138462}}
	\multiput(22.494,20.392)(.0532307692,-.0336891026){312}{\line(1,0){.0532307692}}
	\put(83.354,10.616){\circle*{2.102}}
	\put(113.417,10.616){\circle*{1.734}}
	\put(98.28,38.576){\circle*{1.88}}
	\put(104.587,16.292){\circle*{1.784}}
	\put(95.968,21.337){\circle*{1.784}}
	\multiput(97.65,38.156)(-.0337348837,-.0645348837){430}{\line(0,-1){.0645348837}}
	\put(83.144,10.406){\line(1,0){30.272}}
	\multiput(113.417,10.406)(-.0337126949,.0636770601){449}{\line(0,1){.0636770601}}
	\multiput(105.007,16.293)(.044138122,-.033685083){181}{\line(1,0){.044138122}}
	\put(123.402,10.827){\circle*{2.102}}
	\put(153.465,10.827){\circle*{1.734}}
	\put(138.328,38.787){\circle*{1.88}}
	\put(144.635,16.503){\circle*{1.784}}
	\put(131.812,16.503){\circle*{1.784}}
	\multiput(137.698,38.367)(-.0337348837,-.0645348837){430}{\line(0,-1){.0645348837}}
	\put(123.192,10.617){\line(1,0){30.272}}
	\multiput(153.465,10.617)(-.0337126949,.0636770601){449}{\line(0,1){.0636770601}}
	\multiput(137.698,26.384)(-.033727273,-.055085561){187}{\line(0,-1){.055085561}}
	\multiput(145.056,16.503)(.044132597,-.033679558){181}{\line(1,0){.044132597}}
	\multiput(131.601,16.503)(-.048779006,-.033679558){181}{\line(-1,0){.048779006}}
	\put(138.118,39.208){\line(0,-1){12.614}}
	\put(138.379,26.75){\circle*{1.784}}
	\put(131.811,16.608){\line(1,0){13.139}}
	\multiput(138.118,27.014)(.033570681,-.053931937){191}{\line(0,-1){.053931937}}
	\put(24.491,27.645){\makebox(0,0)[cc]{$x$}}
	\put(99.226,31.534){\makebox(0,0)[cc]{$x$}}
	\put(142.427,22.389){\makebox(0,0)[cc]{$r$}}
	\put(138.013,15.452){\makebox(0,0)[cc]{$p$}}
	\put(133.913,22.81){\makebox(0,0)[cc]{$q$}}
	\put(139.064,32.375){\makebox(0,0)[cc]{$x$}}
	\put(22.074,6){\makebox(0,0)[cc]{$G$}}
	\put(45.934,10.406){\circle*{2.102}}
	\put(75.996,10.406){\circle*{1.734}}
	\put(60.86,38.366){\circle*{1.88}}
	\put(60.229,20.707){\circle*{1.784}}
	\multiput(60.229,37.945)(-.0337325581,-.0645325581){430}{\line(0,-1){.0645325581}}
	\put(45.724,10.196){\line(1,0){30.272}}
	\multiput(75.996,10.196)(-.0337104677,.063674833){449}{\line(0,1){.063674833}}
	\multiput(60.86,38.786)(-.0323077,-1.3745385){13}{\line(0,-1){1.3745385}}
	\multiput(60.44,20.917)(-.0485138462,-.0336369231){325}{\line(-1,0){.0485138462}}
	\multiput(60.019,20.707)(.0532307692,-.0336891026){312}{\line(1,0){.0532307692}}
	\put(62.016,27.96){\makebox(0,0)[cc]{$x$}}
	\multiput(98.28,38.576)(-.03367949,-.22234615){78}{\line(0,-1){.22234615}}
	\multiput(95.653,21.233)(-.0383146417,-.033728972){321}{\line(-1,0){.0383146417}}
	\qbezier(95.863,21.443)(104.009,22.599)(104.377,16.187)
	\qbezier(104.377,16.187)(97.702,14.873)(95.863,21.338)
	\qbezier(60.44,20.812)(61.806,25.963)(64.434,23.545)
	\qbezier(64.434,23.545)(66.011,19.393)(60.44,20.917)
	\qbezier(60.019,20.602)(59.231,26.804)(56.551,23.755)
	\qbezier(56.551,23.755)(54.343,21.548)(60.124,20.602)
	\qbezier(60.124,20.497)(56.603,17.081)(60.019,15.977)
	\qbezier(60.019,15.977)(63.278,17.344)(60.23,20.812)
	\put(65.38,24.386){\makebox(0,0)[cc]{$p$}}
	\put(56.025,25.542){\makebox(0,0)[cc]{$q$}}
	\put(60.23,14.716){\makebox(0,0)[cc]{$r$}}
	\qbezier(95.758,21.233)(94.97,27.435)(92.29,24.386)
	\qbezier(92.29,24.386)(90.082,22.179)(95.863,21.233)
	\put(94.181,27.014){\makebox(0,0)[cc]{$q$}}
	\put(59.284,6){\makebox(0,0)[cc]{$F=G \cup \{p,q,r\}$}}
	\put(96.809,6){\makebox(0,0)[cc]{$F'=F_{\{x,y,p,r\}}$}}
	\put(138.328,6){\makebox(0,0)[cc]{$G''=F'_{\{x,q,r\}}$}}
	\put(18.079,14.821){\makebox(0,0)[cc]{$y$}}
	\put(54.238,13.98){\makebox(0,0)[cc]{$y$}}
	\put(91.658,15.767){\makebox(0,0)[cc]{$y$}}
	\put(102.38,22.389){\makebox(0,0)[cc]{$r$}}
	\put(98.28,15.241){\makebox(0,0)[cc]{$p$}}
	\put(129.394,12.929){\makebox(0,0)[cc]{$y$}}
\end{picture}

\caption{{\it 3-fold} using splitting.}
\label{graph 3fold splitting}
\end{figure}
\noindent Let $\{x,y\}$ be a proper subset of a cutset of $G$ and $G$ be as given in the Figure \ref{graph 3fold splitting}. After adding three loops $\{p,q,r\}$ to $G$ we get graph $F$, later $F'$ is a graph obtained by splitting $F$ using $\{x,y,p,r\}$. Then $G''$ is {\it 3-fold} of $G$ and it is obtained by splitting $F'$ using $\{x,q,r\}$.\\ 
We define {\it 3-fold} in terms of splitting operation as follows.
\begin{defn}\label{defn 3 fold in splitting}
	Let $M$ be a binary matroid and $T=\{x,y\}$ be a proper subset of a cocircuit of $M$. Let $A$ be a matrix representation of $M$. Obtain a matrix $B$ by appending three columns labeled $p, q, r$ containing all zero elements. Let matrix $C=B_{\{x,y,p,r\}}$ and $D=C_{\{x,q,r\}}$. Then the matroid $M(D)$ is {\it 3-fold} of $M$ and it is denoted by $M''$.
\end{defn}
Thus it is very clear that $M''=[[M \cup \{p,q,r\}]_{\{x,y,p,r\}}]_{\{x,q,r\}}$. Let $N=M \cup \{p,q,r\}$, $Y=\{x,y,p,r\}$, $X=\{x,q,r\}$, thus $M''=[N_Y]_X$. In this paper, we find excluded minors for a binary gammoid whose 3-fold $M''$ is a binary gammoid.  Assuming that $N_Y=P$ then we need to find excluded minors for a binary gammoid whose $M''=[N_Y]_X=P_X$ is a binary gammoid. Here $|X|=3$, then by Theorem \ref{thm gam to gam wrt 3 elt in3fold}, $P$ does not contain minors $M(G_i)$, for $i=1,2,3$. Here $M(G_i)_X \cong M(K_4)$, for $X=\{x,q,r\}$, where $G_i$ and $X=\{x,q,r\}$ as shown in Figure \ref{fig gam to gam wrt 3 element in3fold}.\\
Thus  $N_Y \backslash H_1/H_2 =M(G_i)$, for $i=1,2,3$ and for some $H_1, H_2$ subset of $E(N)$. As $\{x,q,r\} \subseteq E(M(G_i))$, for $i=1,2,3$, $q \notin H_1 \cup H_2$, also $q \notin Y$. Thus, $(N_Y \backslash H_1/ H_2)\backslash q =M(G_i) \backslash q$. Thus, $[N\backslash {q}]_Y \backslash H_1 /H_2 = M(F)$, where $F$ is as shown in the Figure \ref{fig minor of pre gammoid}.
\begin{figure}[h]
	\centering
	\unitlength 1mm 
	\linethickness{0.4pt}
	\ifx\plotpoint\undefined\newsavebox{\plotpoint}\fi 
	\begin{picture}(28.264,24.188)(0,0)
		\put(26.81,8.438){\circle*{1.5}}
		\put(16.81,23.438){\circle*{1.5}}
		\put(7.31,8.469){\circle*{1.5}}
		\put(7.158,8.487){\line(1,0){19.534}}
		\multiput(16.595,23.588)(.0336731392,-.0485889968){309}{\line(0,-1){.0485889968}}
		\multiput(16.743,23.291)(-.0336655052,-.0512787456){287}{\line(0,-1){.0512787456}}
		\qbezier(16.892,23.44)(28.264,21.135)(26.852,8.426)
		\qbezier(7.23,8.574)(5,20.764)(16.446,23.44)
		\put(17.25,3){\makebox(0,0)[cc]{$F$}}
	\end{picture}
	
	\caption{Minor}
	\label{fig minor of pre gammoid}
\end{figure}
Let $S=N \backslash q=M \cup \{p,r\}$. Thus our problem reduces to finding excluded minors for a binary gammoid with at least two loops whose splitting using four elements does not contain minor $M(F)$. We use the technique given by Mundhe et al. \cite{gm}, to find excluded minors. \\
In the next section, we outline some results which are used to prove the main theorem.
\section{Preliminary Results}
In this section, we proved preliminary results that are used to find excluded minors for binary gammoid whose splitting using four elements does not contain $M(F)$ as a minor. \\
\noindent {\bf Denote}, $\mathcal{GF}_{k}$ be the collection of binary gammoids whose splitting using $k$ elements contains $M(F)$ as a minor, where $F$ is as given in the Figure \ref{fig minor of pre gammoid}.
	\begin{lem} \label{main lemma in 3fold}
		Let $S$ be a binary gammoid such that $S_Y$ has $M(F)$ a minor and $Y \subseteq E(S)$ with $|Y|\geq 2$, then $S$ has a minor $R$ for which one of the below holds. \\
		i) $R_Y \cong M(F)$.\\
		ii) $R_Y/Y' \cong M(F)$, where $Y' \subseteq Y$.\\
		iii) $R$ has a minor $M(F)$. \\
		iv) $R$ is one element extension of some minimal minor in the collection $\mathcal{GF}_{k-1}$.
	\end{lem}
	\begin{proof}
		Let $S$ be a binary gammoid such that $S_Y$ has a minor $M(F)$. Thus $S_Y \backslash Y_1 /Y_2 \cong M(F)$ for some subsets $Y_1$ and $Y_2$ of $E(S)$. Let $Y_i'=Y \cap Y_i$ and $Y_i''=Y_i-Y_i'$, for $i=1,2$, thus $Y_i' \subseteq Y$ and $Y_i=Y_i' \cup Y_i''$ for $i=1,2$. Then, $S_Y \backslash Y_1 /Y_2 = S_Y \backslash (Y_1' \cup Y_1'') / (Y_2' \cup Y_2'')= [(S \backslash Y_1'' /Y_2'')_Y] \backslash Y_1'/Y_2'  $. Let $R=S \backslash Y_1'' /Y_2''$, hence $R$ is a minor of $S$. Thus $R_Y \backslash Y_1' /Y_2' \cong M(F)$. \\
		If $Y_1'= \phi$ and $Y_2'=\phi$, then (i) holds.\\
		If $Y_1' =\phi$ and $Y_2' \neq \phi$, then (ii) holds.\\
		If $Y_1' \neq \phi$, as $Y_1' \subseteq Y$, $|Y_1|\leq k$. If $|Y_1|=k$, then $Y_1'=Y$ and $Y_2'=\phi$. Thus $R_Y \backslash Y_1 =R \backslash Y=M(F)$. Hence $R$ has a minor $M(F)$, (iii) holds. If $|Y_1'| < k$, let $y \in Y_1'$, $T=Y-y$ and $T'=Y_1'-y$. Then $R_Y \backslash Y_1' /Y_2' =(R\backslash y)_{Y-y} \backslash {Y_1'-y} /Y_2'=M(F)$. Let $R \backslash y=Z$, then $Z_T \backslash T'/Y_2'=M(F)$. Thus $Z$ is a minor of the collection $\mathcal{GF}_{k-1}$ and $Z=R \backslash y$, then $R$ is single element extension of $Z$, hence (iv) holds.  
	\end{proof}
	\begin{lem}\label{rel lemma in 3fold}
		Let $R$ be a binary gammoid as in Lemma \ref{main lemma in 3fold} (i) and (ii). Then there exists a binary gammoid $Q$ and $a \in E(Q)$, such that $Q \backslash a =M(F)$ and $R=Q/a$ or $R$ is a coextension of $Q/a$ by at most $k$ elements.
	\end{lem}
	\begin{proof}
		Let $R$ be a binary gammoid such that $R_Y \cong M(F)$ or $R_Y /Y' \cong M(F)$ for some $Y' \subseteq Y \subseteq E(R)$. From the Definitions \ref{def_splitting_in_3fold} and \ref{def_element_sp_in_3fold} it is very clear that $R_Y'/a =R$ and $R_Y' \backslash a = R_Y$, where $R_Y'$ is an element splitting matroid of $R$ and $R_Y$ is splitting matroid of $R$. \\
		If $R_Y \cong M(F)$, then assume $Q= R_Y'$, $E(Q)=E(R)\cup \{a\}$. Here $Q \backslash a =R_Y' \backslash a = R_Y=M(F)$ and $Q/a=R_Y'/a=R$.\\
		If $R_Y/Y'=M(F)$, then assume $Q=R_Y'/Y'$ then $Q\backslash a=R_Y'/Y' \backslash a=R_Y' \backslash a/Y'=R_Y/Y'=M(F)$ and $Q/a=R_Y'/Y' /a=R_Y'/a /Y'=R/Y'$. Thus $R$ is a coextension of $Q/a$ by $Y'$. As $|Y'| \leq k$, $R$ is a coextension  of $Q/a$ by at most $k$ elements. 
	\end{proof}
	If, for some $a \in E(Q)$, $Q \backslash a =M(F)$ then $Q/a$ is called a quotient of $M(F)$. From Theorem \ref{rel lemma in 3fold}, minor $R$ is either a quotient or a coextension of a quotient. Thus to find excluded minor $R$, we need to find graphic quotients of $M(F)$. Following are a few properties that need to be satisfied by the quotient of $M(F)$.
	\begin{lem}\label{3fold quotient properties}
		Let $Q$ be a binary matroid with $a \in E(Q)$, where $a$ is not a loop or a coloop, such that $Q \backslash a=M(F)$ then $Q/a$ satisfies the following.\\
		i) $Q/a$ contains at most two loops.\\
		ii) $Q/a$ does not contain 2-cocircuit.\\
		iii)  $Q/a$ can not contain more than 4 parallel elements.
	\end{lem} 
	\begin{proof}
		Let $Q$ be a binary matroid with $a \in E(Q)$ and $a$ is not a loop or a coloop, such that $Q\backslash a =M(F)$.\\
		i) Suppose that $Q/a$ has more than two loops, then $Q$ will contain at least a loop or more than 3 parallel elements. Hence, $Q \backslash a$ will also contain at least a loop or more than 3 parallel elements, a contradiction. Hence $Q/a$ can not contain more than two loops. \\
		ii) If $Q/a$ has 2-cocircuit say $\{x,y\}$, then $Q$ will contain 2-cocircuit $\{x,y\}$ and hence $Q\backslash a=M(F)$ has a cocircuit in $\{x,y\}$. Therefore $M(F)$ has a 1-cociruit or a 2-cocircuit, a contradiction. \\
		iii) Suppose $Q/a$ has more than four parallel elements then $Q$ will have at least three parallel elements, hence $Q \backslash a= M(F)$ will have at least three parallel elements, a contradiction.  
	\end{proof}
	\begin{lem}\label{quotient of minor F}
		The graphic quotient of $M(F)$ is isomorphic to $M(Q_i)$, where $Q_i$ is as given in the Figure \ref{fig 3fold quotient}, for $i=1,2,3,4$.
	\end{lem}
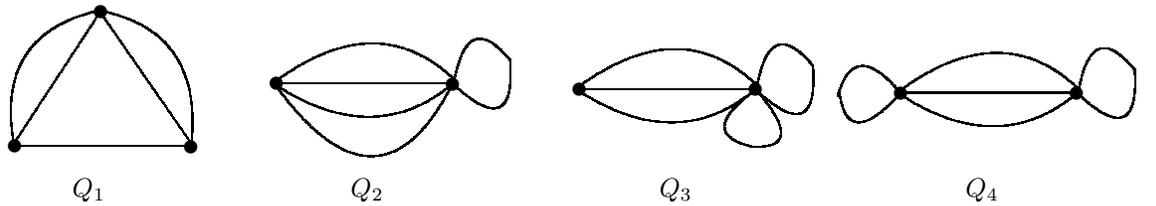
\begin{figure}[h]
\unitlength 1mm 
\linethickness{0.4pt}
\ifx\plotpoint\undefined\newsavebox{\plotpoint}\fi 
\begin{picture}(163,30.226)(0,0)
	\put(36.772,11.326){\circle*{1.8}}
	\put(24.772,29.326){\circle*{1.8}}
	\put(13.372,11.363){\circle*{1.8}}
	\put(13.19,11.384){\line(1,0){23.441}}
	\multiput(24.514,29.506)(.0336549865,-.0485628032){371}{\line(0,-1){.0485628032}}
	\multiput(24.692,29.149)(-.0337046512,-.0513383721){344}{\line(0,-1){.0513383721}}
	\qbezier(24.87,29.328)(38.517,26.562)(36.822,11.311)
	\qbezier(13.276,11.489)(10.6,26.117)(24.335,29.328)
	\put(23.25,5.5){\makebox(0,0)[cc]{$Q_1$}}
	\put(71.55,19.65){\circle*{1.8}}
	\put(48.15,19.687){\circle*{1.8}}
	\put(47.968,19.709){\line(1,0){23.441}}
	\qbezier(47.75,19.75)(61.625,30.125)(72,20)
	\qbezier(72,20)(61.5,10.75)(48,19.5)
	\qbezier(48,19.5)(61.75,.5)(71.5,19.5)
	\qbezier(71.75,19.75)(73.5,29.75)(79.25,22.75)
	\qbezier(79.25,22.75)(80,11.875)(71.75,19.5)
	\put(111.8,18.9){\circle*{1.8}}
	\put(88.4,18.937){\circle*{1.8}}
	\put(88.218,18.959){\line(1,0){23.441}}
	\qbezier(88,19)(101.875,29.375)(112.25,19.25)
	\qbezier(112.25,19.25)(101.75,10)(88.25,18.75)
	\qbezier(112,19)(113.75,29)(119.5,22)
	\qbezier(119.5,22)(120.25,11.125)(112,18.75)
	\put(154.55,18.4){\circle*{1.8}}
	\put(131.15,18.437){\circle*{1.8}}
	\put(130.968,18.459){\line(1,0){23.441}}
	\qbezier(130.75,18.5)(144.625,28.875)(155,18.75)
	\qbezier(155,18.75)(144.5,9.5)(131,18.25)
	\qbezier(154.75,18.5)(156.5,28.5)(162.25,21.5)
	\qbezier(162.25,21.5)(163,10.625)(154.75,18.25)
	\qbezier(122.75,18)(125,25.875)(131.25,18.25)
	\qbezier(123,18.5)(124,10.625)(131,18.25)
	\qbezier(111.75,18.75)(103.875,13)(111.5,11.25)
	\qbezier(111.5,11.25)(118.75,11)(112,18.75)
	\put(60.25,5.5){\makebox(0,0)[cc]{$Q_2$}}
	\put(101.25,5.5){\makebox(0,0)[cc]{$Q_3$}}
	\put(142,5.5){\makebox(0,0)[cc]{$Q_4$}}
\end{picture}

	\caption{Graphic quotients of $M(F)$.}
	\label{fig 3fold quotient}
\end{figure}
\begin{proof}
Let $Q$ be a matroid with $a \in E(Q)$ such that $Q \backslash a =M(F)$.\\
Case i) If $a$ is a loop or a coloop, then $Q/a= Q \backslash a=M(F)=M(Q_1)$.\\
Case ii) If $a$ is not a loop or a coloop, then $r(Q \backslash a)=2$ and $E(Q \backslash a)=5$, thus $r(Q)=2$ and $E(Q)=6$ and hence $r(Q/a)=1$ and $E(Q/a)=5$. Hence $Q/a=M(G)$ and the graph $G$ will have $2$ nodes and $5$ edges. From Lemma \ref{3fold quotient properties}, $Q/a$ does not contain more than two loops, more than four parallel elements, and a 2-cocircuit. Hence, there are the following cases.\\
i) $Q/a$ has four parallel elements and one loop. Hence $Q/a\cong M(Q_2)$ \\
ii) $Q/a$ has three parallel elements and two loops. Hence $Q/a \cong M(Q_3)$ or $Q/a \cong M(Q_4)$.\\
\end{proof}
\section{Splitting of Binary Gammoid}
In this section, we obtain the minors of the collection $\mathcal{GF}_{k}$, where $\mathcal{GF}_{k}$ is the collection of binary gammoids whose splitting using $k$ elements contains minor $M(F)$, for $k \geq 3$.
\begin{lem}
The collection $\mathcal{GF}_{1}$ is empty.
\end{lem}
\begin{proof}
The proof is elementary.
\end{proof}
\begin{lem}
The collection $\mathcal{GF}_{2}$ is empty.
\end{lem}
\begin{proof}
From Lemma \ref{quotient of minor F}, it is clear that the splitting of any quotient using any two elements is not isomorphic $M(F)$. Also, any coextension of quotients does not give $M(F)$ after splitting using any two elements. Hence, the collection $\mathcal{GF}_{2}$ is empty.
\end{proof}
\begin{theorem}\label{thm GF3 in 3fold}
Let $S$ be a binary gammoid, then $S \in \mathcal{GF}_{k}$, for $k\geq 3$ if and only if $S$ has a minor $M(F_i)$, where graph $F_i$ is as shown in Figure \ref{fig GF3 minor}, for $i=1,2,3,4$.
\end{theorem}
\begin{figure}[h]
	\centering
\unitlength 1mm 
\linethickness{0.4pt}
\ifx\plotpoint\undefined\newsavebox{\plotpoint}\fi 
\begin{picture}(163,30.226)(0,0)
	\put(36.772,11.326){\circle*{1.8}}
	\put(24.772,29.326){\circle*{1.8}}
	\put(13.372,11.363){\circle*{1.8}}
	\put(13.19,11.384){\line(1,0){23.441}}
	\multiput(24.514,29.506)(.0336549865,-.0485633423){371}{\line(0,-1){.0485633423}}
	\multiput(24.692,29.149)(-.0337063953,-.0513372093){344}{\line(0,-1){.0513372093}}
	\qbezier(24.87,29.328)(38.517,26.562)(36.822,11.311)
	\qbezier(13.276,11.489)(10.6,26.117)(24.335,29.328)
	\put(23.25,5.5){\makebox(0,0)[cc]{$F_1$}}
	\put(71.55,19.65){\circle*{1.8}}
	\put(48.15,19.687){\circle*{1.8}}
	\put(47.968,19.709){\line(1,0){23.441}}
	\qbezier(47.75,19.75)(61.625,30.125)(72,20)
	\qbezier(72,20)(61.5,10.75)(48,19.5)
	\qbezier(48,19.5)(61.75,.5)(71.5,19.5)
	\qbezier(71.75,19.75)(73.5,29.75)(79.25,22.75)
	\qbezier(79.25,22.75)(80,11.875)(71.75,19.5)
	\put(111.8,18.9){\circle*{1.8}}
	\put(88.4,18.937){\circle*{1.8}}
	\put(88.218,18.959){\line(1,0){23.441}}
	\qbezier(88,19)(101.875,29.375)(112.25,19.25)
	\qbezier(112.25,19.25)(101.75,10)(88.25,18.75)
	\qbezier(112,19)(113.75,29)(119.5,22)
	\qbezier(119.5,22)(120.25,11.125)(112,18.75)
	\put(154.55,18.4){\circle*{1.8}}
	\put(131.15,18.437){\circle*{1.8}}
	\put(130.968,18.459){\line(1,0){23.441}}
	\qbezier(130.75,18.5)(144.625,28.875)(155,18.75)
	\qbezier(155,18.75)(144.5,9.5)(131,18.25)
	\qbezier(154.75,18.5)(156.5,28.5)(162.25,21.5)
	\qbezier(162.25,21.5)(163,10.625)(154.75,18.25)
	\qbezier(122.75,18)(125,25.875)(131.25,18.25)
	\qbezier(123,18.5)(124,10.625)(131,18.25)
	\qbezier(111.75,18.75)(103.875,13)(111.5,11.25)
	\qbezier(111.5,11.25)(118.75,11)(112,18.75)
	\put(60.25,5.5){\makebox(0,0)[cc]{$F_2$}}
	\put(101.25,5.5){\makebox(0,0)[cc]{$F_3$}}
	\put(142,5.5){\makebox(0,0)[cc]{$F_4$}}
	\put(25.75,13.75){\makebox(0,0)[cc]{$x$}}
	\put(32.25,22.5){\makebox(0,0)[cc]{$y$}}
	\put(35,27.5){\makebox(0,0)[cc]{$z$}}
	\put(60.75,28.75){\makebox(0,0)[cc]{$x$}}
	\put(60.25,22){\makebox(0,0)[cc]{$y$}}
	\put(80,25.25){\makebox(0,0)[cc]{$z$}}
	\put(102.5,28){\makebox(0,0)[cc]{$x$}}
	\put(119.5,25.75){\makebox(0,0)[cc]{$y$}}
	\put(113.25,9){\makebox(0,0)[cc]{$z$}}
	\put(128.5,23.5){\makebox(0,0)[cc]{$x$}}
	\put(143.25,27){\makebox(0,0)[cc]{$y$}}
	\put(162.75,24.25){\makebox(0,0)[cc]{$z$}}
\end{picture}

	\caption{Minors of the collection $\mathcal{GF}_k$.}
	\label{fig GF3 minor}
\end{figure}
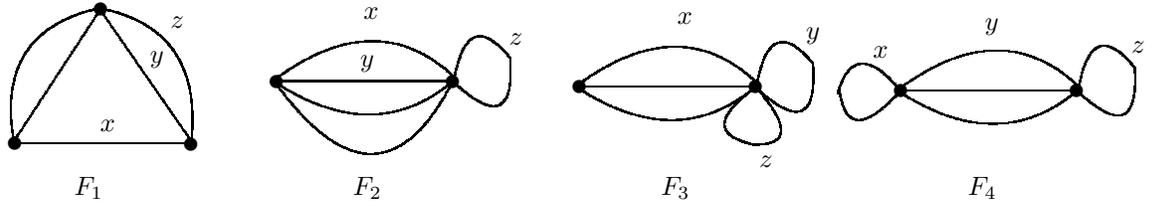
\begin{proof}
Suppose, $S$ has $M(F_i)$ as a minor, then it is very straightforward to prove that $M(F_i)_{\{x,y,z\}}\cong M(F)$, where $\{x,y,z\}$ and $F_i$ is as shown in the Figure \ref{fig GF3 minor}, for $i=1,2,3,4$. Hence $S \in \mathcal{GF}_k$, for $k=3$.\\
Suppose, $M(F_i)$ is not a minor of a binary matroid $S$ then we need to prove that $S_Y$ does not contain $M(F)$ for $|Y| \geq 3$, where $F$ is as shown in Figure  \ref{fig minor of pre gammoid}. On the contrary, assume that $S_Y$ contains $M(F)$ as a minor, then by Lemma \ref{main lemma in 3fold}, $S$ has a minor $R$ such that one of the below holds.\\
i) $R_Y \cong M(F)$.
ii) $R_Y/Y' \cong M(F)$.
iii) $R$ has a minor $M(F)$. 
iv) $R$ is one element extension of some minimal minor in the collection $\mathcal{GF}_{k-1}$.\\
As $\mathcal{GF}_2=\phi$, we discard the case-(iv). Also, if $R$ has minor $M(F)$ and $M(F)=M(F_1)$ a contradiction. Hence we discard the case-(iii). 
Suppose $R$ is satisfying case-(i) and case-(ii), then by Lemma \ref{rel lemma in 3fold}, there exists a matroid $Q$ with $a \in E(Q)$, such that $Q\backslash a = M(F)$ then either $R=Q/a$ or coextension of $Q/a$ by not more than $k$ elements. Then, by Lemma \ref{quotient of minor F}, $Q/a\cong M(Q_i)$ where $Q_i$ is as shown in the Figure \ref{fig 3fold quotient}, for $i=1,2,3,4$. Thus $R=M(Q_i)$ or $R$ is a coextension of $M(Q_i)$. \\
Here $Q_i=F_i$, a contradiction. Hence we discard $Q_i$, for $i=1,2,3,4$. \\
Thus, $S_Y$ does not contain a minor $M(F)$ and hence $S \notin \mathcal{GF}_k$, for any $k \geq 3$. Hence the proof. 
\end{proof}

\section{Main Theorem}
We now obtain the excluded minors for a binary gammoid whose {\it 3-fold} is a binary gammoid. In this section we prove the main theorem hence we state the theorem again.
\begin{thm}
	Let $M$ be a binary gammoid then {\it 3-fold} of $M$ is binary gammoid if and only if $M$ does not contain $M(G_4)$ as a minor, where $G_4$ is shown in the Figure \ref{graph 3fold proof}.
	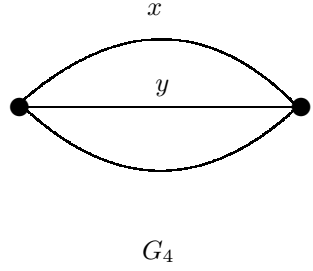
\begin{figure}[h!]
		\centering
		\unitlength 1mm 
		\linethickness{0.4pt}
		\ifx\plotpoint\undefined\newsavebox{\plotpoint}\fi 
	\begin{picture}(56.75,41.625)(0,0)
			\put(18,23.75){\circle*{2.5}}
			\put(55.5,23.75){\circle*{2.5}}
			\put(17.75,23.75){\line(1,0){38.5}}
			\qbezier(17.75,24)(37.625,41.625)(55,23.75)
			\qbezier(55,23.75)(36.75,6.75)(18.5,23.75)
			\put(36,36.75){\makebox(0,0)[cc]{$x$}}
			\put(37,26.25){\makebox(0,0)[cc]{$y$}}
			\put(36.5,4.5){\makebox(0,0)[cc]{$G_4$}}
		\end{picture}
		
		\caption{Excluded minor for a binary gammoid whose {\it 3-fold} is gammoid.}
		\label{graph 3fold proof}
	\end{figure}
\end{thm}
\begin{proof}
	Let $M$ be a binary gammoid containing a minor $M(G_4)$. Then, let $A$ be a matrix representing $M(G_4)$. \\
	
	$ A =\left[ \begin{array}{ccc}
		x & y &    \\
		1 & 1 & 1    
	\end{array} \right].$ 
	Then, $B$ is a matrix obtained by adding three loops $\{p,q,r\}$.\\
	$ B =\left[ \begin{array}{cccccc}
		x & y &  &p &q &r \\
		1 & 1 & 1 &0 &0 &0  
	\end{array} \right].$ Then $C=B_{x,y,p,r}$. Thus, $C$ is as follows.\\
	$ C =\left[ \begin{array}{cccccc}
		x & y &  &p &q &r\\
		1 & 1 & 1 &0 &0  &0\\
		1 & 1 &0&1 &0  &1
	\end{array} \right].$ \\
	Let $D=C_{\{x,q,r\}}$. The matrix $D$ is as follows. \\
	$ D =\left[ \begin{array}{cccccc}
		x & y &  &p &q &r\\
		1 & 1 & 1 &0 &0  &0\\
		1 & 1 &0&1 &0  &1\\
		1 & 0 &0&0 &1  &1
	\end{array} \right].$ \\
	Thus, by Definition \ref{defn 3 fold in splitting}, $M(G_4)''$ using $T=\{x,y\}$ is $M(D)$ and $M(D)\cong M(K_4)$, hence $M(G_4)''$ not a binary gammoid.\\
	Conversely, suppose that $M$ does not contain minor $M(G_4)$, then we need to prove that $M''$ is a binary gammoid. On the contrary, assume that $M''$ is not a binary gammoid. By definition \ref{defn 3 fold in splitting}, $M''=[[M\cup \{p,q,r\}]_{\{x,y,p,r\}}]_{\{x,q,r\}}$ is not gammoid, where $\{p,q,r\}$ are loops. Then, by Theorem \ref{thm gam to gam wrt 3 elt in3fold}, 
	$[M\cup \{p,q,r\}]_{\{x,y,p,r\}}$ has $M(G_i)$ as a minor, where $G_i$ is as shown in the Figure \ref{fig gam to gam wrt 3 element in3fold} for $i=1,2,3$. $[M\cup \{p,q,r\}]_{\{x,y,p,r\}} \backslash H_1 /H_2=M(G_i)$. As $\{x,q,r\}\subseteq E(M(G_i))$, $q \notin H_1 \cup H_2$ and $q \notin Y$. Deleting $\{q\}$ from both side we get $[[M\cup \{p,q,r\}]_{\{x,y,p,r\}}] \backslash H_1 / H_2]\backslash q=M(G_i)\backslash q$. Thus $[M\cup \{p,r\}]_{\{x,y,p,r\}}] \backslash H_1 / H_2=M(F)$. Thus  $[M\cup \{p,r\}]_{\{x,y,p,r\}}$ has a minor $M(F)$. Hence, by Theorem \ref{thm GF3 in 3fold}, $M \cup \{p,r\}$ has a minor $M(F_i)$,where $F_i$ is as shown in the Figure \ref{fig GF3 minor}, for $i=1,2,3,4$. As $\{p,r\}$ are loops then $M\cup \{p,r\} /T_1 \backslash T_2 =M(F_i)$, for $i=3,4$. Deleting $\{p,r\}$ from both sides we get $[M\cup \{p,r\} /T_1 \backslash T_2]\backslash \{p,r\} =M(F_i) \backslash \{p,r\}$, $M/T_1\backslash T_2=M(G_4)$. Thus, $M$ has a minor $M(G_4)$, a contradiction. Hence $M''$ is a binary gammoid. 
\end{proof}
\bibliographystyle{amsplain}

\end{document}